\documentclass[11pt, reqno]{amsart}
\usepackage{indentfirst, amssymb, amsmath, amsthm, mathrsfs, setspace, indentfirst, enumerate,  mathrsfs, amsmath, amsthm}
\usepackage[bookmarksnumbered, colorlinks, plainpages]{hyperref}
\usepackage{mathrsfs}

\newtheorem*{theoA}{Theorem A}
\newtheorem*{theoB}{Theorem B}
\newtheorem*{theoC}{Theorem C}
\newtheorem*{theoD}{Theorem D}
\newtheorem*{theoE}{Theorem E}
\newtheorem*{theoF}{Theorem F}
\newtheorem*{theoG}{Theorem G}
\newtheorem*{theoH}{Theorem H}
\newtheorem*{theoI}{Theorem I}
\newtheorem*{theoJ}{Theorem J}
\newtheorem*{theoK}{Theorem K}

\newtheorem*{conjA}{Conjecture A}
\newtheorem*{conjB}{Conjecture B}

\newtheorem*{QuesA}{Question A}
\newtheorem*{QuesB}{Question B}

\newtheorem*{cor A}{Corollary A}
\newtheorem*{cor B}{Corollary B}
\newtheorem{theo}{Theorem}[section]
\newtheorem{lem}{Lemma}[section]

\newtheorem{exm}{Example}[section]

\newtheorem{rem}{Remark}[section]

\newcommand{\ol}{\overline}
\newcommand{\be}{\begin{equation}}
\newcommand{\ee}{\end{equation}}
\newcommand{\beas}{\begin{eqnarray*}}
\newcommand{\eeas}{\end{eqnarray*}}
\newcommand{\bea}{\begin{eqnarray}}
\newcommand{\eea}{\end{eqnarray}}

\numberwithin{equation}{section}
\begin{document}
\title[U\MakeLowercase{niqueness of entire function concerning derivatives}....]{\LARGE U\Large\MakeLowercase{niqueness of entire function concerning derivatives and shifts}}
\date{}
\author[S. M\MakeLowercase{ajumder} \MakeLowercase{and} N. S\MakeLowercase{arkar}]{S\MakeLowercase{ujoy} M\MakeLowercase{ajumder}$^*$ \MakeLowercase{and} N\MakeLowercase{abadwip} S\MakeLowercase{arkar}}
\address{Department of Mathematics, Raiganj University, Raiganj, West Bengal-733134, India.}
\email{sm05math@gmail.com, sjm@raiganjuniversity.ac.in}

\address{Department of Mathematics, Raiganj University, Raiganj, West Bengal-733134, India.}
\email{naba.iitbmath@gmail.com}

\renewcommand{\thefootnote}{}
\setcounter{footnote}{0}
\footnote{*\emph{Corresponding Author}: Sujoy Majumder.}
\renewcommand{\thefootnote}{\arabic{footnote}}
\setcounter{footnote}{0}

\begin{abstract} In the paper, we investigate the uniqueness problem of entire function concerning its derivative and shift and obtain two results. On of our result solves the open problem posed by Majumder et al. (On a conjecture of Li and Yang, Hiroshima Math. J., 53 (2023), 199-223) and the other result improves and generalizes the recent result due to Huang and Fang (Unicity of entire functions concerning their shifts and derivatives, Comput. Methods Funct. Theory, 21 (2021), 523-532) in a large extend.
\end{abstract}

\thanks{Typeset by \AmS -\LaTeX}
\maketitle
\section{{\bf Introduction and main results}}
In the paper, we assume that the reader is familiar with standard notation and main results of Nevanlinna Theory (see \cite{WKH1, YI1}). We denote respectively by $\rho(f)$ and $\rho_2(f)$ the order and hyper-order of the meromorphic function $f$. As usual, the abbreviation CM means ``counting multiplicities'', while IM means ``ignoring multiplicities''.

\smallskip
We define the linear measure $m(E)$ and the logarithmic measure $l(E)$ respectively by
\[m(E):=\int_E dt\;\;\text{and}\;\;l(E):=\int_{E\cap [1,\infty)} \frac{d t}{t}\]
for a set $E\subset [0,\infty)$. Trivially, $l(E)\leq m(E)$. Also the logarithmic density measure is denoted and defined by
\[\text{log\;dens}\;E=\lim\limits_{r\rightarrow \infty}\frac{l(E(r))}{\log r}=\lim\limits_{r\rightarrow \infty}\frac{\int_{E(r)} (dt/t)}{\log r}\]
where $E(r)=E\cap [e,r]$ for a set $E\subset [0,\infty)$. Moreover, if $l(E)<+\infty$, then $\text{log\;dens}\; E=0$.

A meromorphic function $a$ is said to be a small function of $f$ if $T(r,a)=S(r,f)$ for all $r\not\in E\subset [0,+\infty)$ such that $m(E)<+\infty$.

\smallskip
The research on the uniqueness problem of meromorphic function sharing values or small functions with its derivatives is an active field and the study is based on the Nevanlinna value distribution theory. The research on this topic was started by Rubel and Yang \cite{RY1}. Now we state their result.

\begin{theoA} \cite{RY1} Let $f$ be a non-constant entire function and let $a_1$ and $a_2$ be two finite distinct complex numbers. If $f$ and $f^{(1)}$ share $a_1$ and $a_2$ CM, then $f\equiv f^{(1)}$.
\end{theoA}

This result has been generalized from sharing values CM to IM by Mues and Steinmetz \cite{MS1} and in the case when both shared values are non-zero by Gundersen \cite{GGG1}. 

\smallskip
The research in this topic has been extended in the following directions:
\begin{enumerate}
\item[(i)] One replaces the shared values by small function functions;
\item[(ii)] One replaces sharing CM by IM;
\item[(iii)] One replaces $f^{(1)}$ by $f^{(k)}$.
\end{enumerate}

For this background, we see \cite{GF1}, \cite{XHH1}, \cite{LY1}, \cite{LXY1}, \cite{MSS1}, \cite{GQ1}, \cite{SS1}, \cite{LZY1}, \cite{ZW1}.

\smallskip
In 1992, Zheng and Wang \cite{ZW1} considered shared small functions and improved Theorem A in the following manner. 
\begin{theoB} \cite[Theorem 1]{ZW1} Let $f$ be a non-constant entire function and let $a_1(\not\equiv \infty)$ and $a_2(\not\equiv \infty)$ be two distinct small functions of $f$. If $f$ and $f^{(1)}$ share $a_1$ and $a_2$ CM, then $f\equiv f^{(1)}$.  
\end{theoB}
 
In 2000, Qiu \cite{GQ1} replaced sharing CM to IM in Theorem B and proved the following. 
\begin{theoC} \cite[Theorem 1]{GQ1} Let $f$ be a non-constant entire function and let $a_1(\not\equiv \infty)$ and $a_2(\not\equiv \infty)$ be two distinct small functions of $f$. If $f$ and $f^{(1)}$ share $a_1$ and $a_2$ IM, then $f\equiv f^{(1)}$.
\end{theoC}

On the other hand, Yang \cite{LZY1} investigated the problem of uniqueness of an entire function when it share two values with its $k$-th derivative and obtained the following results.

\begin{theoD} \cite[Theorem 1]{LZY1} Let $f$ be a non-constant entire function, $k\geq 2$ be an integer and let $a_1$ be a non-zero finite complex number. Suppose $0$ is a Picard exceptional value of both $f$ and $f^{(k)}$. If $f$ and $f^{(k)}$ share $a_1$ IM, then $f\equiv f^{(k)}$ and so $f(z)=e^{Az+B}$, where $A$ and $B$ are constants such that $A^k=1$.
\end{theoD}

\begin{theoE} \cite[Theorem 2]{LZY1} Let $f$ be a non-constant entire function, $k\geq 2$ be an integer and let $a_1$ and $a_2$ be two distinct finite complex numbers. If $f$ and $f^{(k)}$ share $a_1$ and $a_2$ CM, then $f\equiv f^{(k)}$.
\end{theoE} 

\smallskip
Frank \cite{GF1} proposed the following conjecture.

\begin{conjA} If a non-constant entire function $f$ shares two finite values IM with its $k$-th derivative, then $f\equiv f^{(k)}$.
\end{conjA}

In 2000, Li and Yang \cite{LY1} fully settled Conjecture A in the following way.

\begin{theoF}\cite[Theorem 2.3]{LY1} Let $f$ be a non-constant entire function and let $a_1$ and $a_2$ be two distinct complex numbers. If $f$ and $f^{(k)}$ share $a_1$ and $a_2$ IM, then $f\equiv f^{(k)}$. 
\end{theoF}

Regarding Theorem F, Li and Yang \cite{LY1} posed the following conjecture at the end of the same paper.

\begin{conjB} Theorem F still holds when $a_1$ and $a_2$ are two arbitrary distinct small functions of $f$.
\end{conjB}

To the knowledge of authors Conjecture B is not still confirmed. Recently Majumder et al. \cite{MSS1} settled Conjecture B partially and obtained the following result.

\begin{theoG}\cite[Theorem 1.3]{MSS1} Let $f$ be a non-constant entire function and let $a_1(\not\equiv \infty)$ and $a_2(\not\equiv \infty)$ be two distinct non-constant small functions of $f$ such that $a_1^{(2)}\not\equiv a_2^{(2)}$.  If $f$ and $f^{(k)}\;(k\geq 1)$ share $a_1$ and $a_2$ IM, then $f\equiv f^{(k)}$.
\end{theoG}

Also in the same paper,  Majumder et al. \cite{MSS1} asked the following question:
\begin{QuesA} Is it possible to establish Theorem G without the hypothesis ``$a_1^{(2)}\not\equiv a_2^{(2)}$'' ?
\end{QuesA}

In the paper, we solve Question A fully. In fact, we prove the following result.

\begin{theo}\label{t2} Let $f$ be a non-constant entire function and let $a_1(\not\equiv \infty)$ and $a_2(\not\equiv \infty)$ be two distinct small functions of $f$ such that $a_1a_2\not\in\mathbb{C}$. If $f$ and $f^{(k)}\;(k\geq 1)$ share $a_1$ and $a_2$ IM, then $f\equiv f^{(k)}$.
\end{theo}

\begin{rem} Following example asserts that condition ``$a_1(\not\equiv \infty)$ and $a_2(\not\equiv \infty)$'' is sharp in \textrm{Theorem \ref{t2}}.
\end{rem}

\begin{exm} Let 
\[f(z)=c+e^{ce^{z}},\;\;a_1(z)=\frac{c^2}{c-e^{-z}}\]
and $a_2=\infty$, where $c\in\mathbb{C}\backslash \{0\}$. Clearly $f$ and $f^{(1)}$ share $a_1$ and $a_1$ CM, but $f\not\equiv f^{(1)}$. 
\end{exm}

\begin{rem} From the proof of Theorem \ref{t2}, we can say that Theorem \ref{t2} holds for meromorphic function having few poles, i.e., $N(r,f)=o(T(r,f))$. But following example asserts that Theorem \ref{t2} does not hold for meromorphic function having infinitely many poles.
\end{rem}

\begin{exm} Let 
\[f(z)=\frac{4}{1-3e^{-2z}}.\]

Note that $N(r,f)\not=S(r,f)$ and  
\[f^{(1)}(z)=\frac{-24e^{-2}}{(1-3e^{-2z})^2}.\]

Clearly $f$ and $f^{(1)}$ share $0$ CM and $2$ IM, but $f\not\equiv f^{(1)}$.
\end{exm}

\medskip
The the time-delay differential equation 
\[f^{(1)}(x) = f(x-k),\]
$k>0$ is well known and extensively studied in real analysis, which have numerous applications ranging from cell growth models to current collection systems for an electric locomotive to wavelets. For a complex variable counterpart, Liu and Dong \cite{11} studied the complex differential-difference equation $f^{(1)}(z)=f(z+c)$, where $c\in\mathbb{C}\backslash \{0\}$.
Recently, many authors have started to consider the sharing values problems of meromorphic functions with their difference operators or shifts. Some results were considered in \cite{CC2}-\cite{FLSY}, \cite{HKT1}-\cite{HKLR1}, \cite{HF1}, \cite{XHH1}, \cite{MP1}, \cite{MSP1}, \cite{QLY1}, \cite{QY1}.

\smallskip
In 2018, Qi et al. \cite{QLY1} first investigated the uniqueness problem related to $f^{(1)}(z)$ and $f(z+c)$ and obtained the following result.

\begin{theoH}\cite[Theorem 1.4]{QLY1} Let $f$ be a finite order transcendental entire function and $a(\neq 0)$ be a finite complex number. If $f^{(1)}(z)$ and $f(z+c)$ share $0$ and $a$ CM, then $f(z+c)\equiv f^{(1)}(z)$.
\end{theoH}

In 2020, Qi and Yang \cite{QY1} improved Theorem H and proved the following results.

\begin{theoI}\cite[Theorem 1.2]{QY1} Let $f$ be a finite order transcendental entire function and $a(\neq 0)$ be a finite complex number. If $f^{(1)}(z)$ and $f(z+c)$ share $0$ CM and $a$ IM, then $f(z+c)\equiv f^{(1)}(z)$.
\end{theoI}

\begin{theoJ}\cite[Theorem 1.4]{QY1} Let $f$ be a finite order transcendental entire function and let $a$ and $b$ be two distinct finite complex numbers. If $f^{(1)}(z)$ and $f(z+c)$ share $a$ and $b$ IM and if $\ol N\left(r,a;f^{(1)}\right)=o(T(r,f))$, then $f(z+c)\equiv f^{(1)}(z)$.
\end{theoJ}

Regarding Theorem J, Huang and Fang \cite{HF1} asked the following question.
\begin{QuesB} Is the condition ``$\ol N(r,a;f^{(1)})=o(T(r,f))$'' in Theorem J necessary or not?
\end{QuesB}

In the same paper, Huang and Fang \cite{HF1} gave the positive answer to Question B. In fact, in the following, they proved more general result.

\begin{theoK}\cite[Theorem 1]{HF1} Let $f$ be a transcendental entire function such that $\rho_2(f)<1$, let $c$ be a non-zero finite complex value and let $a$ and $b$ be two distinct finite values. If $f^{(1)}(z)$ and $f(z+c)$ share $a$ and $b$ IM, then $f(z+c)\equiv f^{(1)}(z)$.
\end{theoK}

\smallskip
In the paper, we have extended and improved Theorem K in the following directions:
\begin{enumerate}
\item[(1)] We replace the first derivative $f^{(1)}$ by the general derivative $f^{(k)}$.
\item[(2)] We consider $a$ and $b$ as the small functions of $f$ in Theorem K.
\end{enumerate}

\smallskip
We now state our next result.

\begin{theo}\label{t1} Let $f$ be a non-constant entire function such that $\rho_2(f)<1$, let $c$ be a non-zero finite complex value and let $a_1(\not\equiv \infty)$ and $a_2(\not\equiv \infty)$ be two distinct small functions of $f$. If $f(z+c)$ and $f^{(k)}(z)$ share $a_1$ and $a_2$ IM, then $f(z+c)\equiv f^{(k)}(z)$.
\end{theo}

\begin{rem}
In the general case that $f(z+c)$ and $f^{(k)}(z)$ have two shared values in Theorem \ref{t1} is necessary. This may be seen by the following example.
\end{rem}

\begin{exm} Let 
\[f(z)=ae^{\alpha z}+b\]
and $e^{\alpha c}=\frac{\alpha(\gamma-b)}{\gamma}$.
Note that 
\[f^{(1)}(z)=a\alpha e^{\alpha z}\;\;\text{and}\;\;f(z+c)=\frac{a\alpha (\gamma -b)}{\gamma}e^{\alpha z}+b\]
and so 
\[f(z+c)-\gamma =\frac{a\alpha (\gamma -b)}{\gamma}\left(e^{\alpha z}-\frac{\gamma }{a\alpha }\right)\;\;\text{and}\;\;f^{(1)}(z)-\gamma =a\alpha\left(e^{\alpha z}-\frac{\gamma }{a\alpha }\right).\]
 
Then $f(z+c)$ and $f^{(1)}(z)$ share $\gamma$ CM, but $f(z+c)\not\equiv  f^{(1)}(z)$.    
\end{exm}

\medskip
We know that if $f$ is a non-constant meromorphic function such that $\rho_2(f)<1$, then 
\bea\label{O1}T(r, f(z))=T(r,f(z+c))+o(T(r,f)),\eea
where $c\in\mathbb{C}\backslash \{0\}$ (see \cite{HKT1}). Clearly (\ref{O1}) shows that $S(r,f(z+c))=o(T(r,f(z)))$.
Now with the help of Lemma \ref{l6} and (\ref{O1}), we get by simple computation that 
\[N(r,f(z))=N(r,f(z+c))+o(T(r,f(z))).\]

Therefore if $N(r,f(z))=S(r,f)$, then 
\[N(r,f(z+c))=o(T(r,f))\;\;\text{and}\;\;N(r,f^{(k)}(z))=o(T(r,f)).\]

Finally from the proof of Theorem \ref{t1}, we can say that Theorem \ref{t1} holds for meromorphic function having few poles. But following example asserts that Theorem \ref{t1} does not hold for meromorphic functions having infinitely many poles.

\begin{exm}\label{ex1} Let 
\[f(z)=\frac{2}{1-e^{-2z}}\]
and $c=\pi \iota$.
Note that 
\[f(z+c)=\frac{2}{1-e^{-2z}}\;\;\text{and}\;\;f^{(1)}(z)=\frac{-4e^{-2z}}{(1-e^{-2z})^2}\]
and so
\[f(z+c)-1=\frac{1+e^{-2z}}{1-e^{-2z}}\;\;\text{and}\;\;f^{(1)}(z)-1=-\left(\frac{1+e^{-2z}}{1-e^{-2z}}\right)^2.\]

Then $f(z+c)$ and $f^{(1)}(z)$ share $0$ CM and $1$ IM, but $f(z+c)\not\equiv f^{(1)}(z)$.
\end{exm}

\subsection{\bf {Notation}}

\medskip
We assume that the reader is familiar with standard notations such as $T(r,f)$, $m(r,a;f)$ $N(r,a;f)$, $\ol N(r,a;f),$ etc of Nevanlinna Theory. Let $\hat{\mathbb{C}}=\mathbb{C}\cup\{\infty\}$ denote the Riemann sphere. For $a\in \hat{\mathbb{C}}$, we put
\[N_1(r,a;f)=N(r,a;f)-\ol N(r,a;f).\]

\medskip
Next we introduce Shimizu and Ahlfors characteristic function.
Let $w$ be the complex coordinate of the finite part $\mathbb{C}$ of the Riemann sphere $\hat{\mathbb{C}}$. We define a surface element on $\hat{\mathbb{C}}$ by
\[\Omega=\frac{1}{(1+|w|^2)^2}\frac{\iota}{2\pi}dw\wedge d \ol w.\]

This is called the Fubini-Study metric form on $\hat{\mathbb{C}}$ and
\[\int_{\hat{\mathbb{C}}}\Omega=1.\]

For a meromorphic function $f(z)$ we define Shimizu's order function $T_f(r,\Omega)$ by
\[T_f(r,\Omega)=\int_{1}^{r}\frac{d t}{t}\int_{\mathbb{C}(t)}f^*\Omega,\]
where $\mathbb{C}(t)=\{z\in\mathbb{C}: |z|<t\}$.

\medskip
First we recall the map $f:(D, |\;\; |_{\mathbb{R}^2})\to (\hat{\mathbb{C}},\chi)$
from $D$ (endowed with the Euclidean metric) to the extended complex plane $\hat{\mathbb{C}}$,
endowed with the chordal metric $\chi$, given by
\beas \chi(z,z')= \left\{\begin{array}{clcr} \frac{|z-z'|}{\sqrt{1+|z|^{2}}\sqrt{1+|z'|^{2}}}, &\;\;\text{if}\;\; z, z'\in\mathbb{C},\\
\frac{1}{\sqrt{1+|z|^{2}}},&\;\;\text{if}\;\;z'=\infty.
\end{array}\right.\eeas

Also we know that $\chi(z,z')\leq |z-z'|$ in $\mathbb{C}$. 
We define the proximity function $m_f(r, a)$ by
\[m_f(r, a)=\frac{1}{2\pi} \int_0^{2\pi} \log \left(\frac{1}{\chi(f(re^{i\theta}),a)}\right) \, d\theta.\]

We note that 
\[\log^+|f(z)|\leq \log \sqrt{1+|f(z)|^2}\leq \log^+|f(z)|+\frac{1}{2}\log 2\]
and so
\bea\label{sn1}m(r,f)\leq m_f(r, \infty)\leq m(r,f)+\frac{1}{2}\log 2.\eea

We recall the first fundamental theorem in the form of Shimizu and Ahlfors: For a meromorphic function $f(z)$, we have
\bea\label{sn2}T_f(r,\Omega)=N(r,a;f)+m_f(r,a)-m_f(1,a),\eea
where $a\in \hat{\mathbb{C}}$. Now by (\ref{sn1}) and (\ref{sn2}), we get
\bea\label{sn3} T(r,f)=T_f(r,\Omega)+O(1)&=&\int_{1}^{r}\frac{d t}{t}\int_{\mathbb{C}(t)}f^*\Omega+O(1)\\&=&
\int_{1}^{r}\frac{d t}{t}\int_{\mathbb{C}(t)}\frac{|f^{(1)}(z)|^2}{(1+|f(z)|^2)^2}\frac{\iota}{2\pi} dz \wedge d \ol z+O(1).\nonumber
\eea

Thus $T(r,f)$ and $T_f(r,\Omega)$ differ by a bounded term and this means that in  most applications they can be used interchangeably.

\medskip
Let $\mathcal{R}_d$ be the set of all rational functions of degree less than or equal to $d$ including the constant function which is identically equal to $\infty$. In 2013, Yamanoi \cite[pp. 706]{KY2} introduced the following modified proximity function
\[\ol{m}_{d,q}(r, f) = \sup_{(a_1, a_2, \dots, a_q)\in (\mathcal{R}_d)^q} \frac{1}{2\pi} \int_0^{2\pi} \max_{1 \leq j \leq q} \log \left( \frac{1}{\chi(f(re^{i\theta}), a_j(re^{i\theta}))} \right) \, d\theta.\]

\medskip
Let $a\in \mathcal{R}_d$. Let $f$ be a meromorphic function with $f\not\in \mathcal{R}_d$. Then by Lemma 2.2 \cite[pp.711]{KY2}, we have
$m_f(1,a)<C$, where $C$ is a positive constant which only depends on $d$ and $f$. It is easy to verify from (\ref{sn2}) and (\ref{sn3}) that
\beas m_f(r,a)=m(r,a;f)+O(1),\eeas
where $a\in \mathcal{R}_d$. Therefore for $(a_1, a_2, \dots, a_q)\in (\mathcal{R}_d)^q$, we have
\bea\label{sn4} \frac{1}{2\pi} \int_0^{2\pi} \max_{1 \leq j \leq q} \log \left( \frac{1}{\chi(f(re^{i\theta}), a_j(re^{i\theta}))} \right) \, d\theta\leq \sideset{}{_{j=1}^q}{\sum} m_f(r,a_j)=\sideset{}{_{j=1}^q}{\sum} m(r,a_j;f)+O(1).\eea

Also by Remark 2.3 \cite[pp.712]{KY2}, we have $\ol{m}_{d,q}(r, f)<+\infty$.

\section {{\bf Auxiliary lemmas}}
The following result is the well known second fundamental theorem for small functions.

\begin{lem}\label{l4}\cite[Corollary 1]{KY1} Let $f$ be a non-constant meromorphic function on $\mathbb{C}$, and let $a_l,\ldots,a_q$ be distinct meromorphic functions on $\mathbb{C}$. Assume that $a_i$ are small functions with respect to $f$ for all $i=1,\ldots,q$. Then we have the second main theorem,
\[(q-2-\varepsilon)\;T(r, f) \leq \sideset{}{_{i=1}^{q}}{\sum}\ol N(r,a_i;f)+\varepsilon T(r, f),\]
for all $\varepsilon>0$ and for all $r\not\in E\subset (0,+\infty)$ such that $\int_E d \log \log r <+\infty.$
\end{lem}

In 2013, Yamanoi \cite{KY2}, obtained the following asymptotic equality.
\begin{lem}\label{l5}\cite[Theorem 1.6]{KY2} Let $f$ be a transcendental meromorphic function and let $\nu: \mathbb{R}_{>e} \to \mathbb{N}_{>0} $ be a function such that
\[\nu(r) \sim \left(\log^+\frac{T_f(r,\Omega)}{\log r}\right)^{20}.\]
Then we have
\[\ol{m}_{0,\nu(r)}(r, f)+\sideset{}{_{a \in \mathbb{\hat C}}}{\sum} N_1(r,a;f)=2T_f(r,\Omega)+o(T_f(r,\Omega)),\]
for all $r\to\infty $ outside a set $E$ of logarithmic density $0$.
\end{lem}

\begin{lem}\label{l6}\cite[Theorem 5.1]{HKT1} Let $f$ be a non-constant meromorphic function such that $\rho_2(f)<1$ and let $c\in\mathbb{C}\backslash \{0\}$. Then
\[m\left(r,f(z+c)/f(z)\right)+m\left(r,f(z)/f(z+c)\right)=o(T(r,f))\]
for all $r$ outside of a possible exceptional set $E$ with finite logarithmic measure.
\end{lem}

\begin{lem}\label{l7} Let $f$ be a non-constant entire function such that $\rho_2(f)<1$, $c$ be a non-zero constant and let $a_1(\not\equiv \infty)$ and $a_2(\not\equiv \infty)$ be two distinct small functions of $f$. If $f(z+c)$ and $f^{(k)}(z)$ share $a_1$ and $a_2$ IM and if $T(r,f(z+c))=T(r,f^{(k)}(z))+o(T(r,f))$, then 
\[f(z+c)\equiv f^{(k)}(z).\]
\end{lem}

\begin{proof} We will prove Lemma \ref{l7} with the idea of proof of Lemma 2.7 \cite{MSS1}.
If possible suppose $f(z+c)\not\equiv f^{(k)}(z)$. Clearly by Lemma \ref{l6}, we have
\beas m\big(r,f^{(k)}(z)/f^{(k)}(z+c)\big)+m\big(r,f^{(k)}(z+c)/f^{(k)}(z)\big)=S(r,f^{(k)})\leq o(T(r,f))\eeas
for all $r\not\in E$ such that $l(E)<+\infty$ and so
\bea\label{al1} m\big(r,f^{(k)}(z)/f(z+c)\big)\leq m\big(r,f^{(k)}(z)/f(z)\big)+m\left(r,f(z)/f(z+c)\right)=o(T(r,f))\eea
for all $r\not\in E$ such that $l(E)<+\infty$. Let $g(z)=f(z+c)$. Obviously $g$ and $f^{(k)}$ share $a_1$ and $a_2$ IM. Now using (\ref{al1}), we get
\bea\label{al2} \sideset{}{_{i=1}^2}{\sum}\ol N(r,a_i;g)&\leq& N(r,0;g-f^{(k)})+o(T(r,f))\\
&\leq& T(r,g-f^{(k)})+o(T(r,f))\nonumber\\
&\leq& m(r,g-f^{(k)})+o(T(r,f))\nonumber\\
&\leq& m(r,g)+m(r, 1-f^{(k)}/g)+o(T(r,f))\nonumber\\
&\leq& T(r,g)+o(T(r,f))\nonumber,\eea
for all $r\not\in E$.
Also by Lemma \ref{l4}, we have $T(r,g)\leq \ol N(r,a_1;g)+\ol N(r,a_2;g)+o(T(r,f))$
and so from (\ref{al2}), we get
\bea\label{al3} T(r,g)= \ol N(r,a_1;g)+\ol N(r,a_2;g)+o(T(r,f)),\eea
for all $r\not\in E$.
Let $\Delta(g)=(g-a_1)(a_1^{(1)}-a_2^{(1)})-(g^{(1)}-a_1^{(1)})(a_1-a_2)$.
It is easy to verify that $\Delta(g)\not\equiv 0$. Since $g\not\equiv f^{(k)}$, so
\bea\label{al4} \phi=\frac{\Delta(g)\left(g-f^{(k)}\right)}{(g-a_1)(g-a_2)}\not\equiv 0.\eea

Also it is easy to prove that $N(r,\phi)=o(T(r,f))$. Note that 
\[\frac{\Delta(g)}{(g-a_1)(g-a_2)}=\frac{1}{a_1-a_2}\left[\frac{\Delta(g)}{g-a_1}-\frac{\Delta(g)}{g-a_2}\right]\;
\text{and}\; 
\frac{\Delta(g)g}{(g-a_1)(g-a_2)}=\frac{\Delta(g)}{g-a_1}+\frac{a_2\Delta(g)}{(g-a_1)(g-a_2)}.\]

Clearly
\bea\label{al6}m\left(r,\frac{\Delta(g)}{(g-a_1)(g-a_2)}\right)=o(T(r,f))\;\;
\text{and}\;\;
m\left(r,\frac{\Delta(g)g}{(g-a_1)(g-a_2)}\right)=o(T(r,f)),\eea
for all $r\not\in E$. Therefore
\beas T(r,\phi)=N(r,\phi)+m(r,\phi)\leq m\left(r,\frac{\Delta(g)g}{g-a_1)(g-a_2)}\right)
+m\left(r,1-\frac{f^{(k)}}{g}\right)+S(r,f)=o(T(r,f)),\eeas
for all $r\not\in E$, which shows that $\phi$ is a small function of $f$.
Also from (\ref{al4}), we have 
\beas \frac{1}{g}=\frac{\Delta(g)}{\phi (g-a_1)(g-a_2)}\left(1-\frac{f^{(k)}}{f}\frac{f}{g}\right)\eeas
and so using (\ref{al1}) and (\ref{al6}), we have $m(r,0;g)=o(T(r,f))$ for all $r\not\in E$.

\medskip
Let $a_3=a_1+l(a_1-a_2)$, where $l$ is a positive integer. If $F=(g-a_1)/(a_2-a_1)$,
then in view of the second fundamental theorem and using (\ref{al3}), we get
\beas 2T(r,g)=2T(r,F)&\leq& \ol N(r,F)+\ol N(r,0;F)+\ol N(r,1;F)+\ol N(r,-l;F)+o(T(r,f))\nonumber\\
&\leq& \ol N(r,a_1;g)+\ol N(r,a_2;g)+\ol N(r,a_3;g)+o(T(r,f))
\\&\leq& 2\;T(r,g)-m(r,a_3;g)+o(T(r,f))\nonumber\eeas
and so $m(r,a_3;g)=S(r,f)$ for all $r\not\in E$. 
Therefore
\bea\label{al9} m(r,0;g)=o(T(r,f))\;\;\text{and}\;\;m(r,a_3;g)=o(T(r,f))\eea
for all $r\not\in E$.
Now proceeding in the same way as done in the proof of Lemma 2.7 \cite{MSS1}, we get a contradiction. 
Hence $g\equiv f^{(k)}$, i.e., $f(z+c)\equiv f^{(k)}(z)$. This completes the proof.
\end{proof}

\begin{lem}\label{l8}\cite{AS1} Let $f$ and $g$ be two non-constant polynomials, and let $a$ and $b$ be two distinct finite values. If $f$ and $g$ share $a$ and $b$ IM, then $f\equiv g$.
\end{lem}

\section {{\bf Proof of Theorem \ref{t1}}} 
\begin{proof}
By the given conditions, $f^{(k)}(z)$ and $f(z+c)$ share $a_1$ and $a_2$ IM.  

\medskip
First we suppose that $f$ is a non-constant polynomial. We know that a small function of a polynomial must be a constant. Therefore $a_1$ and $a_2$ are constants. Clearly $f^{(k)}(z)$ and $f(z+c)$ are also non-constant polynomials. Now by Lemma \ref{l8}, we have $f^{(k)}(z)\equiv f(z+c)$, which contradicts the fact that $f$ is a non-constant polynomial.

\medskip
Next we suppose that $f$ is a transcendental entire function. Let $g(z)=f(z+c)$. Now from the proof of Lemma \ref{l7}, we see that $\Delta(g)\not\equiv 0$. We consider the auxiliary $\phi$ defined by (\ref{al4}).

\smallskip
Now we divide following two cases.

\smallskip
{\bf Case 1.} Let $\phi\equiv 0$. Clearly $g\equiv f^{(k)}$, i.e., $f(z+c)\equiv f^{(k)}(z)$.

\smallskip
{\bf Case 2.} Let $\phi\not\equiv 0$. Obviously $g\not\equiv f^{(k)}$ and from the proof of Lemma \ref{l7}, we see that $\phi$ is a small function of $f$. Here we use the results obtained in (\ref{al3}) and (\ref{al9}), which are irrespective of the relation $T(r,g)=T(r,f^{(k)})+o(T(r,f))$. Now rewriting (\ref{al4}), we get
\bea\label{e3.1} \displaystyle g^{(1)}(g-f^{(k)})=\alpha_{1,2}g^2+\alpha_{1,1}g+\alpha_{1,0}+Q_1,\eea
where
\beas \alpha_{1,2}=\frac{a_1^{(1)}-a_2^{(1)}-\phi}{a_1-a_2},\;
\alpha_{1,1}=a_1^{(1)}-a_1\frac{a_1^{(1)}-a_2^{(1)}}{a_1-a_2}+\frac{(a_1+a_2)\phi}{a_1-a_2},\;
\alpha_{1,0}=-\frac{\phi a_1 a_2}{a_1-a_2}\eeas
and
\beas Q_1=-(a_1^{(1)}-a_2^{(1)})gf^{(k)}/(a_1-a_2)-\left(a_1^{(1)}-a_1(a_1^{(1)}-a_2^{(1)})/(a_1-a_2)\right)f^{(k)}.\eeas

\smallskip
Now we consider following two sub-cases.

\smallskip
{\bf Sub-case 2.1.} Let $\phi\not\equiv a_1^{(1)}-a_2^{(1)}$. Certainly $\alpha_{1,2}\not\equiv 0$. Now differentiating (\ref{e3.1}) and using it repeatedly, we get
\bea\label{e3.3} \displaystyle g^{(k)}(g-f^{(k)})^{2k-1}=\sideset{}{_{j=0}^{2k}}{\sum}\alpha_{k,j}g^j+Q_k,\eea
where
\bea\label{e3.3a}\displaystyle Q_k=\sum\limits_{l<2k,l+j_1+\ldots+j_k\leq 2k}\beta_{l, j_1, j_2,\ldots,j_k}g^{l}(f^{(k)})^{j_1}(f^{(k+1)})^{j_2}\ldots (f^{(2k-1)})^{j_k}.\eea

\smallskip
Obviously $\alpha_{k,j}$ and $\beta_{l, j_1, j_2,\ldots,j_k}$ are small functions of $f$. If we take $\psi_i:=\alpha_{i,2i}$, then
$\psi_1=\alpha_{1,2}$ and $\psi_{i+1}=\psi_i^{(1)}+\psi_1 \psi_i$, where $i=1,2,\ldots,k-1$. Also we have
$\psi_k=\psi_1^k+Q(\psi_1)$, where $Q(\psi_1)$ is a differential polynomial in $\psi_1$ with a degree less than or equal $k-1$.

\smallskip
Now we divide following two sub-cases.

\smallskip
{\bf Sub-case 2.1.1.} Let $\psi_k=\alpha_{k,2k}\not\equiv 0$. Then from (\ref{e3.3}), we have
\bea\label{e3.6} \displaystyle \sideset{}{_{j=0}^{2k}}{\sum}\alpha_{k,j}g^j=g^{(k)}(g-f^{(k)})^{2k-1}-Q_k.\eea

\medskip
Using Lemma \ref{l6} and (\ref{al9}) to (\ref{e3.3a}), we get 
\bea\label{e3.6a} m(r,Q_k/g^{2k-1}g^{(k)})=o(T(r,f)).\eea

\smallskip
Now using Mohon'ko lemma \cite{6}, Lemma \ref{l6} and (\ref{e3.6a}) to (\ref{e3.6}), we get
\beas 2k T(r,g)&=&T\left(r,\sideset{}{_{j=0}^{2k}}{\sum}\alpha_{k,j}g^j\right)+o(T(r,f))\\
&\leq& (2k-1) m\left(r, 1-(f^{(k)}/f)(f/g)\right)+m(r,Q_k/g^{2k-1}g^{(k)})
+m(r,g^{2k-1})\\&&+m(r,f^{(k)})+m(r,g^{(k)}/f^{(k)})+o(T(r,f))\\&
\leq& (2k-1)T(r,g)+T(r,f^{(k)})+o(T(r,f)),\eeas
\;\;i.e., $T(r,g)\leq T(r,f^{(k)})+o(T(r,f))$.
Since $f$ is entire, using Lemma \ref{l6}, we have
\bea\label{ssr} T(r,f^{(k)})=m(r,f^{(k)})=m\left(r,\frac{f^{(k)}}{f}\frac{f}{g}g\right)
\leq m(r,g)+o(T(r,f))=T(r,g)+o(T(r,f)).\eea

Consequently $T(r,g)=T(r,f^{(k)})+o(T(r,f))$ and so by Lemma \ref{l7}, we have $g\equiv f^{(k)}$, which is impossible.

\smallskip
{\bf Sub-case 2.1.2.} Let $\psi_k=\alpha_{k,2k}\equiv 0$. Now proceeding similarly as done in the proof of Sub-case 1.1.2 of Theorem 1.3 \cite{MSS1}, we get a contradiction.

\smallskip
{\bf Sub-case 2.2.} Let $\phi\equiv a_1^{(1)}-a_2^{(1)}$. If $a_1$ and $a_2$ are constants, then $\phi\equiv 0$, which is impossible. Hence atleast one of $a_1$ and $a_2$ is non-constant. Let $\lbrace d_1,d_2,\ldots, d_p\rbrace\subset \mathbb{C}$ such that $d_i\neq a_j$, where $i=1,2,\ldots, p$ and $j=1,2$. Now in view of (\ref{al3}) and using Lemma \ref{l4}, we get 
\beas (p+1-\varepsilon/4)T(r,g)&\leq& \ol N(r,g)+\ol N(r,a_1;g)+\ol N(r,a_2;g)+\sideset{}{_{i=1}^p}{\sum}\ol N(r,d_i;g)+(\varepsilon/4) T(r,g)\\
&\leq &(p+1)T(r,g)-\sideset{}{_{i=1}^p}{\sum} m(r,d_i;g)+(\varepsilon/2)T(r,g),\eeas
for all $\varepsilon>0$ and for all $r\not\in E_1$ such that $\int_{E_1}d \log \log r<\infty$ and so
\bea\label{bm2} \sideset{}{_{i=1}^p}{\sum} m(r,d_i;g)<(\varepsilon/2) T(r,g)\eea
for all $\varepsilon>0$ and for all $r\not\in E_1$.
 
We consider following sub-cases.

\smallskip
{\bf Sub-case 2.2.1.} Let $a_1$ and $a_2$ be non-constant small functions of $f$. Set 
\[S_a=\left\lbrace z\in\mathbb{C}: g(z)=a\;\;\text{and}\;\;g^{(1)}(z)=0\right\rbrace\;\;\text{and}\;\;S_{g^{(1)}}=\left\lbrace z\in\mathbb{C}: g^{(1)}(z)=0\right\rbrace.\]

Obviously
\[\sideset{}{_{a\in\mathbb{C}}}{\bigcup} S_a\subset S_{g^{(1)}}\]
and the set $S_{g^{(1)}}$ is countable. Consequently the set $\sideset{}{_{a\in\mathbb{C}}}{\bigcup} S_a$ is also countable. Clearly there exists a countable set $S\subset \mathbb{C}$ such that $\sideset{}{_{a\in S}}{\bigcup} S_a=\sideset{}{_{a\in\mathbb{C}}}{\bigcup} S_a$. The set may be finite or infinite. For the sake of simplicity we may assume that the set $S$ is infinite. Let $S=\{b_1,b_2,\ldots, b_n,\ldots\}$.

Now in view of (\ref{al3}) and using Lemma \ref{l4}, we get
\bea\label{bmm1} \left(q+1-\frac{\varepsilon}{4}\right)T(r,g)&\leq& \ol N(r,g)+\ol N(r,a_1;g)+\ol N(r,a_2;g)+\sideset{}{_{i=1}^q}{\sum}\ol N(r,b_i;g)+\frac{\varepsilon}{4}T(r,g)\nonumber\\&=&T(r,g)+\sideset{}{_{i=1}^q}{\sum}\ol N(r,b_i;g)+\frac{\varepsilon}{4}T(r,g),\eea
for all $\varepsilon>0$ and for all $r\not\in E_2$ such that $\int_{E_2}d \log \log r<\infty$. Clearly (\ref{bmm1}) yields
\[q\;T(r,g)\leq \sideset{}{_{i=1}^q}{\sum}\ol N(r,b_i;g)+(\varepsilon/2) T(r,g)\]
for all $r\not\in E_2$. By the first fundamental theorem, we get
\[\sideset{}{_{i=1}^q}{\sum} N(r,b_i;g)+\sideset{}{_{i=1}^q}{\sum} m(r,b_i;g)+O(1)\leq \sideset{}{_{i=1}^q}{\sum}\ol N(r,b_i;g)+(\varepsilon/2)T(r,g)\]
for all $r\not\in E_2$ and so by (\ref{bm2}), we conclude that
\bea\label{bmm2}\sideset{}{_{i=1}^q}{\sum} N_1(r,b_i;g)<(\varepsilon/2)T(r,g)\eea
for all $\varepsilon>0$ and for all $r\not\in E_1\cup E_2$. Since (\ref{bmm2}) holds for any finite $q$, we deduce that  
\bea\label{bmm3}\label{bm4} \sideset{}{_{a\in\mathbb{C}}}{\sum}N_1(r,a;g)=\sideset{}{_{i=1}^{\infty}}{\sum} N_1(r,b_i;g)<(\varepsilon/2) T(r,g)\eea
for all $\varepsilon>0$ and for all $r\not\in E_1\cup E_2$.
Let $\nu: \mathbb{R}_{>e} \to \mathbb{N}_{>0}$ be a function such that
\[\nu(r) \sim \left(\log^+\frac{T_g(r,\Omega)}{\log r}\right)^{20}.\]

By (\ref{sn3}), we know that $T(r,g)=T_g(r,\Omega)+O(1)$. Since $g$ is a transcendental, we have
\[\nu(r) \sim \left(\log^+\frac{T_g(r,\Omega)}{\log r}\right)^{20}=o(T(r,g)).\]

Let $q=v(r)$ be a positive integer. Then for $(c_1, c_2, \dots, c_q)\in (\hat {\mathbb{C}})^q$, we get from (\ref{sn4}) that
\bea\label{bm6}\frac{1}{2\pi} \int_0^{2\pi} \max_{1 \leq j \leq q} \log \left( \frac{1}{\chi(g(re^{i\theta}), c_j)} \right) \, d\theta\leq \sideset{}{_{j=1}^q}{\sum} m(r,c_j;g)+O(1).\eea

Clearly from (\ref{bm2}) and (\ref{bm6}), we get
\[\frac{1}{2\pi} \int_0^{2\pi} \max_{1 \leq j \leq q} \log \left( \frac{1}{\chi(g(re^{i\theta}), b_j)} \right) \, d\theta\leq m(r,g)+(\varepsilon/2)T(r,g)\]
and so
\bea\label{bm7}\ol{m}_{0,q}(r,g)&=&\sup_{(c_1, c_2, \dots, c_q)\in (\hat {\mathbb{C}})^q} \frac{1}{2\pi} \int_0^{2\pi} \max_{1 \leq j \leq q} \log \left( \frac{1}{\chi(g(re^{i\theta}), c_j)} \right) \, d\theta
\\&\leq& m(r,g)+(\varepsilon/2)T(r,g)
\nonumber\eea
for all $\varepsilon>0$. Now by Lemma \ref{l5}, we have 
\bea\label{bm8} \ol{m}_{0,q}(r,g)+\sideset{}{_{a \in \mathbb{\hat C}}}{\sum} N_1(r,a;g)=2T(r,g)+o(T(r,g)),\eea
for all $r\not\in E_3$ such that $\text{log dens}\;E_3=0$.
Let $E=E_1\cup E_2\cup E_3$. Then $\text{log dens}\;E=0$. Since $f$ is an entire function, from (\ref{bm4}), (\ref{bm7}) and (\ref{bm8}), we get
\beas 2T(r,g)&=&\ol m_{0,q}(r,g)+\sideset{}{_{a\in\mathbb{C}}}{\sum} N_1(r,a;g)+o(T(r,g))\\
&\leq & m(r,g)+\varepsilon T(r,g)+o(T(r,g))\\
&\leq & T(r,g)++\varepsilon T(r,g)+o(T(r,g)),\eeas
for all $\varepsilon>0$ and for all $r\not\in E$. Clearly $T(r, g)=o(T(r, g))$, for all $r\not\in E$. So we get a contradiction.

\smallskip
{\bf Sub-case 2.2.2.} Let $a_1$ be a non-constant small function and $a_2$ be a finite complex number. Then from (\ref{bm4}), we get
\bea\label{bm8a} \sideset{}{_{a\in\mathbb{C}\backslash \{a_2\}}}{\sum}N_1(r,a;g)<(\varepsilon/2) T(r,g)\eea
for all $0<\varepsilon>0$ and for all $r\not\in E_1\cup E_2$.
Let $q=v(r)$ be a positive integer. Then for $(c_1, c_2, \dots, c_q)\in (\hat {\mathbb{C}})^q$, we get from (\ref{bm2}) and (\ref{bm6})  that
\[\frac{1}{2\pi} \int_0^{2\pi} \max_{1 \leq j \leq q} \log \left( \frac{1}{\chi(g(re^{i\theta}), c_j)} \right) \, d\theta\leq m(r,g)+m(r,a_2;g)+(\varepsilon/2)T(r,g)\]
and so
\bea\label{bm11}\ol{m}_{0,q}(r,g)&=&\sup_{(c_1, c_2, \dots, c_q)\in (\hat {\mathbb{C}})^q} \frac{1}{2\pi} \int_0^{2\pi} \max_{1 \leq j \leq q} \log \left( \frac{1}{\chi(g(re^{i\theta}), c_j)} \right) \, d\theta
\\&\leq& m(r,g)+m(r,a_2;g)+(\varepsilon/2)T(r,g).\nonumber\eea

Now from (\ref{bm8})-(\ref{bm11}), we get 
\beas 2T(r,g)&=&\ol{m}_{0,q}(r,g)+\sideset{}{_{a\in\mathbb{C}}}{\sum} N_1(r,a;g)+o(T(r,g))\\
&\leq& m(r,g)+m(r,a_2;g)+N_1(r,a_2;g)+\varepsilon T(r,g)+o(T(r,g))\\
&\leq & 2T(r,g)-\ol N(r,a_2;g)+\varepsilon T(r,g)+o(T(r,g)),\eeas
for all $\varepsilon>0$ and for all $r\not\in E$, which shows that $\ol N(r,a_2;g)=o(T(r,g))$ for all $r\not\in E$.

Note that $\ol N(r,a_1;g)=\ol N(r,a_1;f^{(k)})+o(T(r,g))$. Now in view of (\ref{al3}) and using the first fundamental theorem, we get
\beas T(r,g)\leq \ol N(r,a_1;g)+o(T(r,g))\leq T(r,f^{(k)})+o(T(r,g)).\eeas

Consequently from (\ref{ssr}), we get $T(r,g)=T(r,f^{(k)})+o(T(r,f))$ and so by Lemma \ref{l7}, we have $g\equiv f^{(k)}$, which is impossible.

\smallskip
{\bf Sub-case 2.2.3.} Let $a_2$ be a non-constant small function and $a_1$ be a finite complex number. Now proceeding similarly as done in the proof of Sub-case 2.2.2, we get a contradiction.
Hence the proof.
\end{proof}

\section {{\bf Proof of Theorem \ref{t2}}}
\begin{proof} We prove Theorem \ref{t2} with the line of proof of Theorem \ref{t1}, where $g(z)=f(z)$ and we use Lemma 2.7 \cite{MSS1} instead of Lemma \ref{l7}. So we omit the detail.
\end{proof}

\medskip
{\bf Compliance of Ethical Standards:}

\medskip
{\bf Conflict of Interest.} The authors declare that there is no conflict of interest regarding the publication of this paper.

\medskip
{\bf Data availability statement.} Data sharing not applicable to this article as no data sets were generated or analysed during the current study.

\end{document}